\documentclass[11pt]{article}
\usepackage[utf8]{inputenc}
\usepackage[T1]{fontenc}
\usepackage{amsmath,amssymb,amsthm}
\usepackage[margin=2.6cm]{geometry}

\theoremstyle{plain}
\newtheorem{theorem}{Theorem}[section]
\newtheorem{proposition}[theorem]{Proposition}
\newtheorem{lemma}[theorem]{Lemma}
\newtheorem{corollary}[theorem]{Corollary}
\theoremstyle{definition}
\newtheorem{definition}[theorem]{Definition}
\theoremstyle{remark}
\newtheorem{remark}[theorem]{Remark}
\newtheorem{example}[theorem]{Example}
\newtheorem{note}[theorem]{Note}

\newcommand{\SDS}{\mathsf{SDS}}

\title{ Algebra of spectral duality structures}
\author{J.~A.~Vega Coso \\
        \small IUFFyM, Universidad de Salamanca \\
        \small Salamanca, Spain}
\date{}

\begin{document}
\maketitle

\begin{abstract}
This paper establishes the algebraic and categorical foundations of
spectral duality structures (SDS). We show that the class of all SDS
admits the structure of a graded monoidal category, where the degree
\(K\) is the fundamental invariant indexing the strata \(\SDS_K\).
We define operations of Cartesian product and disjoint union,
morphisms preserving involution and weights, and prove that
\(C^\ast=1/(1+\sqrt K)\) is a functor constant on each stratum.
We classify objects up to isomorphism by the combinatorial type
\((k,f)\), the degree \(K\), and a multiset of inversion classes
\([r]=\{r,r^{-1}\}\). Furthermore, we prove that the response rank
is the rank of an explicit matrix built from the amplitudes and the
weights, and hence exactly the number of non-trivial pairs \(k\)
under the full-rank condition, and we connect the
categorical structure with the Fisher--Rao geometry of Paper~V.
This work shows that the SDS, originally discovered in stochastic
resetting problems \cite{EvansMajumdarSchehr2020}, is an autonomous
mathematical object with a rich algebraic structure and a natural
geometric realisation.
\end{abstract}

\section{Introduction}

The first five papers of this series \cite{PaperI,PaperII,PaperIII,PaperIV,PaperV} studied, from different perspectives, a \emph{spectral duality structure} (SDS) associated with geometric resetting processes \cite{EvansMajumdarSchehr2020}. They showed that the SDS determines a universal invariant \(C^\ast\), a separatrix \(\Sigma\) of neutral distributions, and a rigid Fisher--Rao geometry on the distribution simplex. However, each of those works focused on a \emph{single} SDS (or on a parametrised family, such as the biased walk). The phenomenon from which the programme started --- the exact invariance of the ruin probability from the midpoint --- is not tied to the geometric mechanism: Paper~I shows that it holds for every resetting mechanism independent of the walk. The SDS is the structure that this invariance acquires in the spectral representation of the geometric case, and it is the abstract structure, not its probabilistic origin, that is studied here. The natural next step, and the one we take here, is to consider the \emph{class of all SDS} and ask: what operations close the class? How does the invariant behave under them? What global algebraic structure emerges?

We propose that the class of SDS admits a complete description as a \emph{graded monoidal category}. The invariant \(K\) (and hence \(C^\ast\)) provides a grading by the multiplicative group \(\mathbb R_{>0}\), so that the Cartesian product of two SDS of degrees \(K_1\) and \(K_2\) has degree \(K_1K_2\), while the disjoint union is only possible within the same degree. We define the natural morphisms between SDS as maps of sites that commute with the involutions and preserve the weights; we prove that these morphisms compose, that \(C^\ast\) is functorial, and that the resulting category has finite coproducts in each degree.Moreover, the degree yields a natural surjective homomorphism from
the Grothendieck group to the multiplicative group
\(\mathbb R_{>0}\). This gives an algebraic interpretation of the
invariant \(K\) as the \emph{degree character} of the category.
The degree is not a complete invariant of the Grothendieck group;
the full description of the semiring of isomorphism classes and of
its Grothendieck constructions is the subject of a companion paper
\cite{PaperVIII}. We also observe that, in each degree, skeletons
are exactly the finite \(C_2\)-sets over the additive group
\((\mathbb R,+)\) with \(C_2\) acting by \(t\mapsto-t\)
(Remark~\ref{rem:equivariant}); this explains conceptually why the
class carries a symmetric monoidal structure, and relates the
classification to the decomposition of equivariant sets into orbits.

A conceptual issue arises when defining these objects. A complete SDS includes, besides the sites \(Z\), the involution \(\sigma\), the weights \(\kappa\), and a system of spectral amplitudes \(\{A_\nu\}\) satisfying a full-rank condition. However, as we shall see, the algebraic operations, morphisms, grading, and invariant \(C^\ast\) depend \emph{only} on the triple \((Z,\sigma,\kappa)\), which we call the \emph{weighted dual skeleton}. The amplitudes \(\{A_\nu\}\) provide a non-degenerate spectral realisation of this skeleton, but do not enter the categorical structure. In fact, every weighted dual skeleton admits at least one canonical spectral realisation (taking the amplitudes as a basis of \(\mathbb R^{|Z|}\)), so the full-rank condition imposes no additional restrictions on the class of objects. This distinction is crucial because it separates the \emph{algebra of duality} (universal and combinatorial) from the \emph{spectral realisation} (analytic and model-dependent). The abstract SDS requires no probabilities or stochastic processes; it is a purely algebraic object that can be realised, among others, by Markovian resetting processes, but has a life of its own beyond them.

Paper~V showed that, when an SDS is realised probabilistically, it
induces a rigid geometry on the distribution simplex: the separatrix
\(\Sigma\) is a totally geodesic submanifold for the Fisher--Rao
metric. In the present paper we show that several of the structural
relations underlying this separatrix --- its orbit-ratio description,
its behaviour under suitable morphisms, and its compatibility with
Cartesian products --- admit a natural formulation at the level of
the abstract SDS. Thus the probabilistic Fisher--Rao geometry
provides a natural geometric realisation of part of the abstract
structure.

The paper is organised as follows. In Section 2 we fix notation and recall the definition of SDS, clearly distinguishing between the complete structure and its skeleton, and proving that every skeleton admits a spectral realisation. In Section 3 we prove the parametrisation of weights by \((K,r)\), which separates the invariant from the deformation directions. In Section 4 we introduce the operations of Cartesian product and disjoint union, and prove that the class is graded by \(K\), with coproducts in each fibre \(\SDS_K\). In Section 5 we define morphisms, prove that they form a category and that \(C^\ast\) is functorial. In Section 6 we study the symmetric monoidal structure, define the
Grothendieck group, and prove that the degree induces a surjective
homomorphism onto \(\mathbb R_{>0}\). In Section 7 we give a complete classification of SDS up to isomorphism. In Section 8 we prove that the response rank is the rank of an explicit matrix built from the amplitudes and the weights, equal to the number of non-trivial pairs \(k\) under the full-rank condition. In Section 9 we connect the results with the Fisher--Rao geometry of Paper~V. Finally, in Section 10 we point out open directions.

\section{Definitions and notation}
\label{sec:definitions}

In this section we fix terminology and basic objects. We distinguish two levels: the \emph{weighted dual skeleton}, which contains purely combinatorial and algebraic data, and the \emph{complete spectral duality structure}, which adds amplitudes and the full-rank condition.

\begin{definition}[Weighted dual skeleton]
\label{def:skeleton}
A \emph{weighted dual skeleton} is a triple \((Z,\sigma,\kappa)\) where:
\begin{enumerate}
\item \(Z\) is a non-empty finite set, whose elements are called \emph{sites}.
\item \(\sigma:Z\to Z\) is an involution (\(\sigma^2=\mathrm{id}\)).
\item \(\kappa:Z\to\mathbb R_{>0}\) is a weight function satisfying
  \begin{equation}
  \kappa(z)\,\kappa(\sigma(z)) = K \quad \text{for all } z\in Z,
  \label{eq:skeleton}
  \end{equation}
  with \(K>0\) constant (independent of \(z\)).
\end{enumerate}
We call \(K\) the \emph{degree invariant} of the skeleton. The number of non-trivial pairs of \(\sigma\) is denoted by \(k\), and the number of fixed points by \(f\); thus \(|Z|=2k+f\).
\end{definition}

\begin{remark}
For a fixed point \(z=\sigma(z)\), condition \eqref{eq:skeleton} implies \(\kappa(z)^2=K\), and since \(\kappa(z)>0\), one necessarily has \(\kappa(z)=+\sqrt K\). This fact will be used repeatedly.
\end{remark}

\begin{definition}[Complete spectral duality structure (SDS)]
\label{def:sds}
A \emph{complete spectral duality structure} (SDS) is a pair \((\mathcal E, \{A_\nu\})\), where \(\mathcal E=(Z,\sigma,\kappa)\) is a weighted dual skeleton and \(\{A_\nu\}_{\nu=1}^N\) is a family of functions \(A_\nu:Z\to\mathbb R\) (the \emph{amplitudes}) satisfying:
\begin{enumerate}
\item The full-rank condition: the \(m=|Z|\) vectors
  \[
  \mathbf A(z)=\bigl(A_1(z),\dots,A_N(z)\bigr)\in\mathbb R^N,
  \qquad z\in Z,
  \]
  are linearly independent in \(\mathbb R^N\) (in particular \(N\ge m\)).
\item The amplitudes define, via the relation
  \begin{equation}
  B_\nu(z)=\kappa(z)\,A_\nu(\sigma(z)),
  \label{eq:amplitudes}
  \end{equation}
  the dual functions \(B_\nu\), which together with the \(A_\nu\) are used in the spectral representation of coupling functionals (see \cite{PaperIII} for details).
\end{enumerate}
\end{definition}

\begin{proposition}[Universal realizability of skeletons]
\label{prop:realizability}
Every weighted dual skeleton \((Z,\sigma,\kappa)\) admits at least one complete SDS. In particular, the existence of full-rank amplitudes imposes no additional restrictions on the class of objects we consider.
\end{proposition}

\begin{proof}
Let \(m=|Z|\). Choose an enumeration \(Z=\{z_1,\dots,z_m\}\) and set \(N=m\), with
\[
A_\nu(z_j)=\delta_{\nu j}.
\]
Then the vectors \(\mathbf A(z_j)=e_j\in\mathbb R^m\) are linearly independent. The dual amplitudes are defined by \eqref{eq:amplitudes}, which is always possible. Hence every triple \((Z,\sigma,\kappa)\) admits a complete SDS.
\end{proof}

\begin{note}[Role of the amplitudes]
\label{note:amplitudes}
The amplitudes \(\{A_\nu\}\) and the full-rank condition are necessary for connecting the SDS to concrete spectral realisations (for instance, the eigenmodes of a Markov generator). However, as we shall prove in the following sections, the algebraic operations (product, sum, morphisms), the grading, and the invariant \(C^\ast\) depend \emph{only} on the skeleton \((Z,\sigma,\kappa)\). The amplitudes do not enter the definition of a morphism. The compatibility of a morphism between skeletons with specific spectral realisations, when required, would need additional data and is not incorporated into the basic category defined in this paper. Therefore, the category \(\SDS\) we study has as objects all weighted dual skeletons, and amplitudes are considered auxiliary structures that, when they exist, provide a non-degenerate realisation.
\end{note}

\begin{definition}[Universal invariant]
\label{def:Cstar}
Given a weighted dual skeleton \((Z,\sigma,\kappa)\) of degree \(K\), we define the \emph{universal invariant}
\[
C^\ast = \frac{1}{1+\sqrt K}.
\]
By the previous remark, \(C^\ast\) is well defined and depends only on \(K\). When the SDS is realised probabilistically, \(C^\ast\) is the value of the ruin probability at neutral distributions (see \cite{PaperIII}).
\end{definition}

\begin{definition}[Orbit ratios]
\label{def:ratios}
For each site \(z\in Z\), we define the \emph{orbit ratio}
\[
r_z = \left(\frac{\kappa(\sigma(z))}{\kappa(z)}\right)^{1/2}.
\]
One immediately checks that
\[
r_{\sigma(z)} = r_z^{-1}.
\]
Thus a non-trivial orbit is parametrised by a single positive number \(r_O\), once one of its two elements is chosen as representative. At a fixed point, \(r_z=1\).
\end{definition}

\begin{remark}[Grading and fibration]
\label{rem:fibration}
The definition of the skeleton clearly separates two types of data:
\begin{itemize}
\item The invariant \(K\) (or equivalently \(C^\ast\)) is a global scalar that \emph{grades} the class: as we shall see in Section 5, two skeletons with different \(K\) cannot be connected by a morphism (the proof is immediate from weight preservation).
\item The orbit ratios \(r_z\) parametrise, for each fixed \(K\), the different ways of distributing the weight \(\kappa\) within each pair. They are thus coordinates on the \emph{fibre} over \(K\).
\end{itemize}
This dichotomy is the basis of the whole categorical structure we develop.
\end{remark}

\begin{note}[On notation]
In what follows, unless explicitly stated otherwise, we write \(\mathcal S\) for a complete SDS and \(\mathcal E=(Z,\sigma,\kappa)\) for its skeleton. When we say ``the category \(\SDS\)'', we refer to the category whose objects are weighted dual skeletons and whose morphisms are maps preserving \(\sigma\) and \(\kappa\) (defined in Section 5). This choice is justified by Proposition~\ref{prop:realizability} and Note~\ref{note:amplitudes}.
\end{note}

\section{Parametrisation by \((K,r)\) and stratification}
\label{sec:parametrisation}

In the previous section we defined the weighted dual skeleton
\((Z,\sigma,\kappa)\) and saw that its degree \(K\) and its orbit
ratios \(r_z\) are fundamental data. We now prove that, in fact,
these two objects completely determine the weight \(\kappa\), and the
correspondence is bijective once a choice of representatives for the
non-trivial orbits is fixed. This cleanly separates the
\emph{global invariant} \(K\) from the \emph{internal coordinates}
of each orbit, laying the foundations for the graded structure
developed in the following sections.

\begin{proposition}[Parametrisation of weights]
\label{prop:param}
Fix a finite set \(Z\), an involution \(\sigma\), a choice of
representatives for the non-trivial orbits, and a number \(K>0\).
Let \((Z,\sigma,\kappa)\) be a weighted dual skeleton of degree \(K\).
For each non-trivial orbit \(O\), with representative \(z_O\), define
\[
r_O := r_{z_O} = \sqrt{\frac{\kappa(\sigma(z_O))}{\kappa(z_O)}}.
\]
Then for each non-trivial orbit \(O=\{z_O,\sigma(z_O)\}\), the weight
\(\kappa\) is determined by \(K\) and \(r_O\) via
\begin{equation}
\kappa(z_O) = \frac{\sqrt K}{r_O}, \qquad
\kappa(\sigma(z_O)) = \sqrt K \, r_O.
\label{eq:kappaKr}
\end{equation}
At a fixed point, one necessarily has \(\kappa(z)=\sqrt K\) and
\(r_z=1\). The map
\[
\kappa \longmapsto (r_O)_{O \text{ non-trivial orbit}}
\]
is a bijection between the set of weights of degree \(K\) on \(Z\)
and the product \((\mathbb R_{>0})^k\), where \(k\) is the number of
non-trivial pairs of \(\sigma\).
\end{proposition}

\begin{proof}
Fix a non-trivial orbit \(O=\{z_O,\sigma(z_O)\}\). The two equations
that the pair \((\kappa(z_O),\kappa(\sigma(z_O)))\) must satisfy are:
\begin{equation}
\kappa(z_O)\,\kappa(\sigma(z_O)) = K,
\qquad
\frac{\kappa(\sigma(z_O))}{\kappa(z_O)} = r_O^2.
\label{eq:system}
\end{equation}
From the second, \(\kappa(\sigma(z_O)) = r_O^2\,\kappa(z_O)\).
Substituting into the first:
\[
\kappa(z_O)^2 \, r_O^2 = K \quad\Longrightarrow\quad
\kappa(z_O) = \frac{\sqrt K}{r_O},
\]
since both factors are positive. Then
\(\kappa(\sigma(z_O)) = r_O^2 \cdot \frac{\sqrt K}{r_O} = \sqrt K\, r_O\).
Thus the pair is uniquely determined by \(K\) and \(r_O\).

For a fixed point \(z_0=\sigma(z_0)\), the condition
\(\kappa(z_0)^2=K\) and positivity give \(\kappa(z_0)=\sqrt K\), and
then \(r_{z_0}=1\). Hence there is no additional freedom.

Conversely, given a collection of positive numbers \(\{r_O\}\), one
for each non-trivial orbit, formula \eqref{eq:kappaKr} together with
\(\kappa(z_0)=\sqrt K\) for fixed points defines a weight
\(\kappa\) that automatically satisfies
\(\kappa(z)\kappa(\sigma(z))=K\) and has the prescribed orbit ratios.
The correspondence is therefore bijective.
\end{proof}

\begin{remark}[Choice of representatives]
\label{rem:representatives}
In the proposition above we fixed a representative \(z_O\) for each
non-trivial orbit. If the other site is chosen, \(r_O\) is replaced
by its inverse, since \(r_{\sigma(z_O)} = r_O^{-1}\). The assignment
of weights to the two sites of the orbit is the same; only the order
of the two components is interchanged. Hence the intrinsic data of
an orbit is the equivalence class \([r_O] = \{r_O, r_O^{-1}\}\).
In what follows, we assume a choice of representatives for each
orbit, although the results are independent of it.
\end{remark}

\begin{remark}[Stratification by degree]
\label{rem:stratification}
The proposition above has a very clear structural reading. Define the
\emph{degree map}
\[
\deg:\{\text{skeletons}\}\longrightarrow \mathbb R_{>0},
\qquad \deg(Z,\sigma,\kappa)=K.
\]
This map decomposes the class of skeletons into \emph{strata}
\[
\SDS_K := \{(Z,\sigma,\kappa): \deg(Z,\sigma,\kappa)=K\}.
\]
Each stratum \(\SDS_K\) has the following structure:
\begin{itemize}
\item The universal invariant \(C^\ast = 1/(1+\sqrt K)\) is constant
  on \(\SDS_K\).
\item The orbit ratios \(r_O\) are coordinates within each stratum:
  they parametrise the different ways of distributing the weight
  \(\kappa\) between the two sites of each pair, keeping the product
  \(K\) fixed.
\end{itemize}
This separation of variables is the basis of the entire categorical
structure we develop: \textbf{\(K\) grades, \(r\) coordinates the
stratum}.
\end{remark}

\begin{corollary}[Structure of the stratum with fixed \((Z,\sigma)\)]
\label{cor:strata}
Given a finite set \(Z\), an involution \(\sigma\), and a choice of
representatives for the non-trivial orbits, the set of weights
\(\kappa\) satisfying \(\kappa(z)\kappa(\sigma(z))=K\) is in
bijection with \((\mathbb R_{>0})^k\), where \(k\) is the number of
non-trivial pairs of \(\sigma\). Hence the stratum \(\SDS_K\)
decomposes as a disjoint union of these subsets, indexed by the
combinatorial type \((Z,\sigma)\).
\end{corollary}

\begin{remark}[Data relevant for classification up to isomorphism]
\label{rem:classification-data}
Up to isomorphism of skeletons, the type of an object of degree \(K\)
is determined by:
\begin{enumerate}
\item The combinatorial type \((k,f)\), where \(k\) is the number of
  non-trivial pairs and \(f\) the number of fixed points (with
  \(2k+f=|Z|\)).
\item A multiset of equivalence classes \(\{[r_1],\ldots,[r_k]\}\),
  where each \([r_i] = \{r_i, r_i^{-1}\}\) is the inversion class of
  a non-trivial orbit.
\end{enumerate}
This notation absorbs both the inversions \(r_i \leftrightarrow r_i^{-1}\)
and the permutations of orbits. The precise classification is
developed in Section 7.
\end{remark}

\begin{remark}[Consequence for category theory]
\label{rem:catfibre}
The bijection of Proposition~\ref{prop:param} implies that a morphism
between skeletons (which preserves \(\kappa\)) must preserve \(K\).
For the orbits, preservation of the involution and of the weights
imposes direct restrictions on the action of a morphism. For example,
if a morphism \(\phi:Z\to Z'\) identifies two sites of the same
non-trivial orbit, this is only possible when \(r_O=1\). Indeed, if
\(\phi(z)=\phi(\sigma(z))\), compatibility with the involutions,
\(\phi(\sigma(z))=\sigma'(\phi(z))\), implies that \(\phi(z)\) is a
fixed point of \(\sigma'\). Weight preservation,
\(\kappa'(\phi(z))=\kappa(z)\), then forces
\(\kappa(z)=\kappa(\sigma(z))\), hence \(r_O=1\). The general
description of these compatibilities, including non-injective cases,
is analysed in detail in Section 5, where we define morphisms and
prove that \(C^\ast\) is functorial.
\end{remark}

\begin{remark}[Structural decomposition and auxiliary role of the spectral realisation]
\label{rem:decomposition}
Section 2 and the present section reveal that the intrinsic SDS
admits a decomposition into two layers:
\[
\boxed{
\text{Intrinsic SDS}
=
\underbrace{(Z,\sigma)}_{\text{combinatorial structure}}
+
\underbrace{(K,(r_O))}_{\text{weight data}}
}
\]
where:
\begin{itemize}
\item The \emph{combinatorial structure} \((Z,\sigma)\) is purely
  discrete and determines the number of orbits and their type.
\item The \emph{weight data} \((K,(r_O))\) include the global degree
  \(K\) (which is invariant) and the internal coordinates \(r_O\)
  (which parametrise the stratum).
\end{itemize}
The spectral realisation \(\{A_\nu\}\), on the other hand, is an
auxiliary datum:
\[
\boxed{
\text{auxiliary spectral realisation } \{A_\nu\}
}
\]
which provides a connection to concrete models (such as the
eigenmodes of a Markov generator) but does not enter the definitions
of morphisms or operations, as will be shown in the following
sections. This separation is fundamental: the SDS as a categorical
object is independent of the choice of amplitudes, as long as at
least one realisation exists (which, by Proposition~\ref{prop:realizability},
always does).
\end{remark}

\section{Operations on the class}
\label{sec:operations}

In this section we define the two fundamental operations that close
the class of weighted dual skeletons: the Cartesian product and the
disjoint union. We prove that both are internal and that, together
with the degree map \(\deg\), they endow the class with a graded
structure that will be the basis of the monoidal category developed
in the following sections.

\begin{definition}[Cartesian product of skeletons]
\label{def:product}
Given two skeletons \(\mathcal E_1=(Z_1,\sigma_1,\kappa_1)\) and
\(\mathcal E_2=(Z_2,\sigma_2,\kappa_2)\), we define their
\emph{Cartesian product}
\[
\mathcal E_1 \times \mathcal E_2 = (Z_1 \times Z_2,\; \sigma_1 \times \sigma_2,\; \kappa_1 \otimes \kappa_2),
\]
where:
\begin{enumerate}
\item The involution on the product is the map
  \[
  (\sigma_1 \times \sigma_2)(z,w) = (\sigma_1(z), \sigma_2(w)).
  \]
\item The weight is the pointwise product
  \[
  (\kappa_1 \otimes \kappa_2)(z,w) = \kappa_1(z)\,\kappa_2(w).
  \]
\end{enumerate}
\end{definition}

\begin{proposition}[The product is internal and computes degree and ratios]
\label{prop:product}
The Cartesian product of two skeletons is again a weighted dual
skeleton. Moreover:
\begin{enumerate}
\item The degree is multiplicative:
  \[
  \deg(\mathcal E_1 \times \mathcal E_2) = K_1 K_2,
  \qquad K_i = \deg(\mathcal E_i).
  \]
\item The orbit ratios of the product are the products of the ratios:
  \[
  r_{(z,w)} = r_z \, r_w.
  \]
\item The universal invariant \(C^\ast\) is additive in logarithms:
  \[
  \operatorname{logit} C^\ast(\mathcal E_1 \times \mathcal E_2)
  = \operatorname{logit} C^\ast(\mathcal E_1)
    + \operatorname{logit} C^\ast(\mathcal E_2),
  \qquad
  \operatorname{logit} C^\ast = -\frac12 \log K.
  \]
  Equivalently, in terms of probabilities:
  \[
  C^\ast_{12}
  =
  \frac{C^\ast_1 C^\ast_2}
       {C^\ast_1 C^\ast_2 + (1-C^\ast_1)(1-C^\ast_2)}.
  \]
\end{enumerate}
\end{proposition}

\begin{proof}
We verify the skeleton condition. For every \((z,w)\in Z_1\times Z_2\),
\[
(\kappa_1 \otimes \kappa_2)(z,w)\,
(\kappa_1 \otimes \kappa_2)(\sigma_1(z),\sigma_2(w))
=
\kappa_1(z)\kappa_1(\sigma_1(z))\,\kappa_2(w)\kappa_2(\sigma_2(w))
= K_1 K_2,
\]
which is constant. Hence \(\mathcal E_1\times\mathcal E_2\) is a
skeleton of degree \(K_1K_2\).

For the orbit ratios, a direct computation gives:
\[
r_{(z,w)}^2
=
\frac{(\kappa_1\otimes\kappa_2)(\sigma_1(z),\sigma_2(w))}
     {(\kappa_1\otimes\kappa_2)(z,w)}
=
\frac{\kappa_1(\sigma_1(z))}{\kappa_1(z)}
\frac{\kappa_2(\sigma_2(w))}{\kappa_2(w)}
= r_z^2\, r_w^2.
\]
Taking positive square roots, \(r_{(z,w)} = r_z r_w\). This formula
holds for every site \((z,w)\), regardless of whether each factor is
in a non-trivial orbit or a fixed point (for which \(r=1\)).

Finally, the formula for \(C^\ast\) follows from
\(C^\ast = 1/(1+\sqrt K)\) and the multiplicativity of \(K\):
\[
\frac{C^\ast_{12}}{1-C^\ast_{12}} = \frac{1}{\sqrt{K_1K_2}}
= \frac{1}{\sqrt{K_1}}\frac{1}{\sqrt{K_2}}
= \frac{C^\ast_1}{1-C^\ast_1}\frac{C^\ast_2}{1-C^\ast_2},
\]
which is the identity in odds. Solving gives the explicit formula.
\end{proof}

\begin{definition}[Disjoint union of skeletons]
\label{def:sum}
Given two skeletons \(\mathcal E_1=(Z_1,\sigma_1,\kappa_1)\) and
\(\mathcal E_2=(Z_2,\sigma_2,\kappa_2)\), we define their
\emph{disjoint union}
\[
\mathcal E_1 \sqcup \mathcal E_2 = (Z_1 \sqcup Z_2,\; \sigma_1 \sqcup \sigma_2,\; \kappa_1 \sqcup \kappa_2),
\]
where:
\begin{enumerate}
\item The set of sites is the disjoint union.
\item The involution acts as \(\sigma_i\) on each component.
\item The weight is \(\kappa_i\) on each component.
\end{enumerate}
\end{definition}

\begin{proposition}[The disjoint union is internal only in fixed degree]
\label{prop:sum}
The disjoint union \(\mathcal E_1 \sqcup \mathcal E_2\) is a weighted
dual skeleton if and only if
\[
\deg(\mathcal E_1) = \deg(\mathcal E_2).
\]
In that case, the common degree is \(K\), and the orbit ratios are
those of each component.
\end{proposition}

\begin{proof}
The skeleton condition requires that
\(\kappa(z)\kappa(\sigma(z))\) be constant on \(Z_1\sqcup Z_2\).
On the first component this constant is \(K_1\); on the second,
\(K_2\). For a global constant, it is necessary and sufficient that
\(K_1=K_2\). If this holds, the componentwise defined weight
satisfies the condition. The orbit ratios are those of each
component, since orbits do not mix in the disjoint union.
\end{proof}

\begin{corollary}[Graded structure of the class]
\label{cor:grading}
The class of weighted dual skeletons has a graded structure
compatible with the two operations:
\begin{enumerate}
\item The degree map \(\deg\) decomposes the class into
  \emph{strata} \(\SDS_K = \{\mathcal E: \deg(\mathcal E)=K\}\).
\item The Cartesian product is a monoidal operation connecting
  strata:
  \[
  \times: \SDS_{K_1} \times \SDS_{K_2} \longrightarrow \SDS_{K_1K_2}.
  \]
  It is associative and commutative (up to natural isomorphisms), and
  has a unit: the skeleton \(\mathcal I\) on a single fixed point
  with \(K=1\).
\item The disjoint union is an internal operation in each stratum:
  \[
  \sqcup: \SDS_K \times \SDS_K \longrightarrow \SDS_K,
  \]
  which provides a coproduct within each stratum (as will be seen in
  Section 5).
\item The universal invariant \(C^\ast\) is additive in logarithms
  under the product:
  \[
  \operatorname{logit} C^\ast(\mathcal E_1\times\mathcal E_2)
  = \operatorname{logit} C^\ast(\mathcal E_1)
    + \operatorname{logit} C^\ast(\mathcal E_2).
  \]
\end{enumerate}
In summary, the class of skeletons is a \emph{graded monoidal
category} (formalised in Section 6), where the degree \(K\) is the
invariant indexing the strata and the Cartesian product is the
tensor product.
\end{corollary}

\begin{remark}[Extension to complete SDS]
\label{rem:amplitudes-ext}
We have defined the operations on skeletons. If complete SDS are
available, i.e. with amplitudes \(\{A_\nu\}\), the operations extend
naturally:
\begin{itemize}
\item For the product, the amplitudes of the product are the tensor
  products of the amplitudes of each factor: if
  \(\mathcal S_i = (\mathcal E_i, \{A^{(i)}_{\nu_i}\})\),
  then
  \[
  A_{(\nu_1,\nu_2)}(z,w) = A^{(1)}_{\nu_1}(z)\, A^{(2)}_{\nu_2}(w).
  \]
  The full-rank condition, in standard cases (e.g. when the
  amplitudes of each factor form full-rank matrices), is inherited by
  the tensor product. However, as emphasised in Note~\ref{note:amplitudes},
  amplitudes are auxiliary to the categorical structure. We do not
  need to develop the general theory of extension of amplitudes here;
  it suffices to know that it exists and is compatible with the
  operations on skeletons.
\item For the disjoint union, the amplitudes are defined
  componentwise: on the first block we use those of \(\mathcal S_1\),
  and on the second those of \(\mathcal S_2\).
\end{itemize}
This separation is one of the conceptual advantages of the
architecture we have established: the fundamental operations of the
category are defined on skeletons, and amplitudes, if they exist,
transform in an induced way.
\end{remark}

\begin{remark}[Interpretation in terms of parametrisation]
\label{rem:param-interp}
In terms of the parametrisation \((Z,\sigma) + K + (r_O)\) of
Section 3, the operations take a particularly simple form:
\begin{itemize}
\item The Cartesian product combines the combinatorial structures
  \((Z_1,\sigma_1)\times(Z_2,\sigma_2)\), multiplies the degrees:
  \(K_{12}=K_1K_2\). For the orbit ratios, the pointwise formula
  \(r_{(z,w)}=r_z r_w\) holds for every site. However, when
  translating to complete orbits, care is needed: if
  \(O_1=\{z,\sigma_1z\}\) and \(O_2=\{w,\sigma_2w\}\) are non-trivial
  orbits, the product \(O_1\times O_2\) is not a single orbit of the
  product involution, but splits into two orbits:
  \[
  \{(z,w),(\sigma_1z,\sigma_2w)\},
  \qquad
  \{(z,\sigma_2w),(\sigma_1z,w)\}.
  \]
  With natural choices of representatives, the corresponding ratios
  are \(r_1r_2\) and \(r_1/r_2\).
\item The disjoint union combines the combinatorial structures
  \((Z_1,\sigma_1)\sqcup(Z_2,\sigma_2)\), keeps the same degree
  \(K\), and the list of coordinates is the concatenation of the
  lists of each factor.
\end{itemize}
This interpretation reinforces the idea that the degree \(K\) is the
true global invariant, while the \(r_O\) are internal coordinates
that behave multiplicatively or additively (in logarithms) according
to the operation, with the subtlety that the product of non-trivial
orbits splits into two distinct orbits.
\end{remark}

\section{Morphisms and category}
\label{sec:morphisms}

We have defined the objects of our theory (weighted dual skeletons)
and the operations combining them. We now introduce morphisms, the
arrows relating objects. The natural choice, consistent with the
skeleton structure, is to require that a morphism preserve both the
involution and the weights.

\begin{definition}[Morphism of skeletons]
\label{def:morphism}
Given two skeletons \(\mathcal E_1=(Z_1,\sigma_1,\kappa_1)\) and
\(\mathcal E_2=(Z_2,\sigma_2,\kappa_2)\), a \emph{morphism}
\(\phi:\mathcal E_1\to\mathcal E_2\) is a map
\(\phi:Z_1\to Z_2\) satisfying:
\begin{enumerate}
\item \emph{Commutation with involutions:}
  \[
  \phi(\sigma_1(z)) = \sigma_2(\phi(z)) \quad \forall z\in Z_1.
  \]
\item \emph{Exact weight preservation:}
  \[
  \kappa_1(z) = \kappa_2(\phi(z)) \quad \forall z\in Z_1.
  \]
\end{enumerate}
\end{definition}

\begin{remark}[Immediate consequences]
\label{rem:consequences}
From the definition several fundamental properties follow:
\begin{enumerate}
\item \emph{Preservation of degree:} \(K_1=K_2\), since
  \[
  K_1=\kappa_1(z)\kappa_1(\sigma_1(z))
  =\kappa_2(\phi(z))\kappa_2(\sigma_2(\phi(z)))=K_2.
  \]
  Hence morphisms only exist between skeletons of the same stratum
  \(\SDS_K\).

\item \emph{Pointwise preservation of ratios:} For every
  \(z\in Z_1\),
  \[
  r_{1,z}^2
  =
  \frac{\kappa_1(\sigma_1(z))}{\kappa_1(z)}
  =
  \frac{\kappa_2(\phi(\sigma_1(z)))}{\kappa_2(\phi(z))}
  =
  \frac{\kappa_2(\sigma_2(\phi(z)))}{\kappa_2(\phi(z))}
  =
  r_{2,\phi(z)}^2.
  \]
  Since the ratios are positive, we obtain the exact equality
  \[
  \boxed{r_{1,z}=r_{2,\phi(z)}}.
  \]
  There is no alternative of inversion here: the inverse appears only
  when translating this equality to orbit coordinates \(r_O\),
  depending on the choice of representatives. If the representative
  is changed \(z\leftrightarrow\sigma(z)\), then
  \(r_O\leftrightarrow r_O^{-1}\).

\item \emph{Behaviour on orbits:} The image of an orbit of
  \(\sigma_1\) is an orbit of \(\sigma_2\) (possibly a fixed point).
  In particular:
  \begin{itemize}
  \item If \(\phi(z)=\phi(\sigma_1(z))\), then \(\phi(z)\) is a
    fixed point of \(\sigma_2\), and weight preservation forces
    \(\kappa_1(z)=\kappa_1(\sigma_1(z))\), i.e. \(r_{1,z}=1\).
  \item If \(\phi(z)\neq\phi(\sigma_1(z))\), then, by commutation
    with the involutions, the image is a non-trivial orbit of
    \(\sigma_2\). The ratios are preserved (or inverted, according
    to the choice of representatives) in the sense explained in the
    previous point.
  \end{itemize}
\end{enumerate}
\end{remark}

\begin{proposition}[Skeletons form a category]
\label{prop:category}
Weighted dual skeletons, with morphisms as defined in
Definition~\ref{def:morphism}, form a category. Composition is
composition of maps, and the identity on \(\mathcal E\) is the
identity map \(\mathrm{id}_Z\).
\end{proposition}

\begin{proof}
The composition of two morphisms is a morphism: if
\(\phi:\mathcal E_1\to\mathcal E_2\) and
\(\psi:\mathcal E_2\to\mathcal E_3\) preserve \(\sigma\) and
\(\kappa\), then \(\psi\circ\phi\) also does:
\[
(\psi\circ\phi)(\sigma_1(z))
=
\psi(\sigma_2(\phi(z)))
=
\sigma_3(\psi(\phi(z))),
\]
and
\[
\kappa_1(z)=\kappa_2(\phi(z))=\kappa_3(\psi(\phi(z))).
\]
The identity \(\mathrm{id}_Z\) trivially satisfies the conditions.
The associativity and identity laws are inherited from the category
of sets.
\end{proof}

\begin{definition}[Graded subcategories]
\label{def:subcat}
For each \(K>0\), we define the full subcategory \(\SDS_K\) whose
objects are skeletons of degree \(K\). By Remark~\ref{rem:consequences},
morphisms only exist within each \(\SDS_K\). Hence
\[
\SDS = \bigsqcup_{K>0} \SDS_K
\]
is a disjoint union of categories.
\end{definition}

\begin{remark}[Canonical function of degree]
\label{rem:Cstar-functor}
The invariant \(C^\ast\) is, in the structure we have built, a
function of the degree:
\[
C^\ast(K) = \frac{1}{1+\sqrt K}.
\]
This function is constant on each stratum \(\SDS_K\) and defines a
functor
\[
F:\SDS\longrightarrow \mathbb R_{>0}^{\mathrm{disc}},
\qquad
F(\mathcal E)=C^\ast(K),
\qquad
F(\phi)=\mathrm{id}_{C^\ast},
\]
to the discrete category of positive reals. The probabilistic
interpretation of \(C^\ast\) as a ruin probability is a particular
realisation, not part of the definition of SDS. The abstract theory
conceptually precedes that interpretation.
\end{remark}

\begin{remark}[Coproducts in \(\SDS_K\)]
\label{rem:coproducts}
Within each \(\SDS_K\), the disjoint union \(\sqcup\) satisfies the
universal property of coproduct. The inclusions
\(i_1:\mathcal E_1\to\mathcal E_1\sqcup\mathcal E_2\) and
\(i_2:\mathcal E_2\to\mathcal E_1\sqcup\mathcal E_2\) are morphisms.
Given morphisms \(f_1:\mathcal E_1\to\mathcal T\) and
\(f_2:\mathcal E_2\to\mathcal T\) with \(\mathcal T\in\SDS_K\), there
exists a unique morphism \(h:\mathcal E_1\sqcup\mathcal E_2\to\mathcal T\)
defined by \(h|_{Z_1}=f_1\), \(h|_{Z_2}=f_2\). (The empty object,
if included, would be the initial object.)
\end{remark}

\begin{remark}[Morphisms and coordinates \(r_O\)]
\label{rem:morphisms-r}
For general morphisms (not necessarily injective), the most
fundamental description is pointwise weight preservation. The
consequences for orbit ratios are derived from that.
In particular:
\begin{itemize}
\item If \(\phi\) is injective and bijects an orbit \(O_1\) of
  \(\sigma_1\) with an orbit \(O_2\) of \(\sigma_2\), then, with a
  suitable choice of representatives, \(r_{O_1}=r_{O_2}\)
  (or \(r_{O_1}=r_{O_2}^{-1}\) if representatives are swapped).
  This is the translation to orbit coordinates of the pointwise
  equality \(r_{1,z}=r_{2,\phi(z)}\).
\item If \(\phi\) collapses a non-trivial orbit to a fixed point,
  then \(r_{O_1}=1\).
\item For non-injective maps identifying sites from different orbits,
  the description is in fact complete: a morphism acts orbit by orbit.
  Its restriction to a non-trivial orbit \(O_1\) is either a bijection
  onto a non-trivial orbit with the same ratio (for the induced choice
  of representatives), or, when \(r_{O_1}=1\), possibly a collapse onto
  a fixed point; a fixed point goes to a fixed point. Distinct orbits
  may have the same image. Conversely, any such orbit-wise assignment
  is a morphism (Remark~\ref{rem:equivariant}).
\end{itemize}
The description in terms of the coordinates \(r_O\) as an "action" on
the set of parameters is valid for isomorphisms and for injective
morphisms preserving the orbit structure, but not for general
morphisms. Hence, in this section we have focused on the pointwise
formulation, which has universal validity.
\end{remark}

\section{Graded monoidal structure}
\label{sec:monoidal}

We have established that \(\SDS\) is a category whose objects are
weighted dual skeletons, and that the Cartesian product \(\times\) is
an associative and commutative operation on objects, with unit the
skeleton \(\mathcal I\) on a single fixed point with \(\kappa=1\), of
degree \(K=1\). In this section we prove that \(\times\) extends to a
bifunctor, endowing \(\SDS\) with the structure of a symmetric
monoidal category graded by degree. Moreover, we study the
Grothendieck group of this monoidal category and show that the degree
induces a surjective homomorphism onto \(\mathbb R_{>0}\). A complete
description of the Grothendieck group is given in the companion
paper \cite{PaperVIII}.

\begin{definition}[Extension of product to morphisms]
\label{def:bifunctor}
Given morphisms \(\phi_1:\mathcal E_1\to\mathcal E_1'\) and
\(\phi_2:\mathcal E_2\to\mathcal E_2'\), we define their
\emph{Cartesian product}
\[
\phi_1 \times \phi_2 : \mathcal E_1 \times \mathcal E_2
\longrightarrow \mathcal E_1' \times \mathcal E_2'
\]
by
\[
(\phi_1 \times \phi_2)(z,w) = (\phi_1(z), \phi_2(w)).
\]
\end{definition}

\begin{proposition}[The product is a bifunctor]
\label{prop:bifunctor}
The assignment
\[
(\mathcal E_1,\mathcal E_2)\longmapsto \mathcal E_1\times\mathcal E_2,
\qquad
(\phi_1,\phi_2)\longmapsto \phi_1\times\phi_2
\]
defines a bifunctor
\[
\times : \SDS \times \SDS \longrightarrow \SDS.
\]
That is, it respects composition and identities in each variable.
\end{proposition}

\begin{proof}
First, \(\phi_1\times\phi_2\) is a morphism. For every
\((z,w)\in Z_1\times Z_2\),
\[
(\phi_1\times\phi_2)(\sigma_1(z),\sigma_2(w))
=
(\phi_1(\sigma_1(z)), \phi_2(\sigma_2(w)))
=
(\sigma_1'(\phi_1(z)), \sigma_2'(\phi_2(w)))
=
(\sigma_1'\times\sigma_2')(\phi_1(z),\phi_2(w)),
\]
and
\[
(\kappa_1\otimes\kappa_2)(z,w)
=
\kappa_1(z)\kappa_2(w)
=
\kappa_1'(\phi_1(z))\,\kappa_2'(\phi_2(w))
=
(\kappa_1'\otimes\kappa_2')(\phi_1(z),\phi_2(w)).
\]
Thus \(\phi_1\times\phi_2\) preserves involution and weights.

Compatibility with composition is immediate:
\[
(\psi_1\circ\phi_1)\times(\psi_2\circ\phi_2)
=
(\psi_1\times\psi_2)\circ(\phi_1\times\phi_2),
\]
and identities are preserved:
\[
\mathrm{id}_{\mathcal E_1} \times \mathrm{id}_{\mathcal E_2}
=
\mathrm{id}_{\mathcal E_1\times\mathcal E_2}.
\]
Hence \(\times\) is a bifunctor.
\end{proof}

\begin{theorem}[\(\SDS\) is a symmetric monoidal category]
\label{thm:monoidal}
The triple \((\SDS,\times,\mathcal I)\), where \(\mathcal I\) is the
skeleton on a single fixed point with \(\kappa=1\), is a symmetric
monoidal category. The natural isomorphisms of associativity, unit,
and commutativity are induced by the canonical set isomorphisms:
\begin{align*}
\alpha_{\mathcal E_1,\mathcal E_2,\mathcal E_3}
&: (\mathcal E_1\times\mathcal E_2)\times\mathcal E_3
\longrightarrow \mathcal E_1\times(\mathcal E_2\times\mathcal E_3),
\\
\lambda_{\mathcal E}
&: \mathcal I\times\mathcal E \longrightarrow \mathcal E,
\\
\rho_{\mathcal E}
&: \mathcal E\times\mathcal I \longrightarrow \mathcal E,
\\
\gamma_{\mathcal E_1,\mathcal E_2}
&: \mathcal E_1\times\mathcal E_2
\longrightarrow \mathcal E_2\times\mathcal E_1,
\end{align*}
which act as the corresponding set bijections and preserve
involutions and weights.
\end{theorem}

\begin{proof}
The verification is routine. The canonical set isomorphisms commute
with the product involutions and preserve pointwise weights, hence
they are morphisms in \(\SDS\). The coherence diagrams (pentagon for
associativity, triangle for unit, and hexagon for commutativity) are
inherited from the category of sets. Therefore,
\((\SDS,\times,\mathcal I)\) is a symmetric monoidal category.
\end{proof}

\begin{corollary}[Grading of the monoidal structure]
\label{cor:monoidal-graded}
The monoidal category \((\SDS,\times,\mathcal I)\) is graded by the
multiplicative group \(\mathbb R_{>0}\) via the degree map \(\deg\).
Specifically:
\begin{enumerate}
\item The tensor product sends objects of degrees \(K_1\) and \(K_2\)
  to degree \(K_1K_2\):
  \[
  \times : \SDS_{K_1} \times \SDS_{K_2}
  \longrightarrow \SDS_{K_1K_2}.
  \]
\item Morphisms preserve degree, so the bifunctor respects the
  grading.
\item The unit \(\mathcal I\) has degree \(1\).
\end{enumerate}
Thus \(\SDS\) is a \emph{graded monoidal category} over the
multiplicative group of positive reals.
\end{corollary}

\begin{remark}[Grothendieck group and degree]
\label{rem:grothendieck}
The Grothendieck group \(K_0(\SDS)\) of the monoidal category
\((\SDS,\times,\mathcal I)\) is, by definition, the abelian group
generated by the isomorphism classes \([\mathcal E]\) of objects of
\(\SDS\), with the relation
\[
[\mathcal E_1] + [\mathcal E_2] = [\mathcal E_1 \times \mathcal E_2],
\]
and with formal inverses for each object.

The degree map \(\deg:\SDS\to\mathbb R_{>0}\) is a monoidal
homomorphism (by multiplicativity of degree), and is constant on
isomorphism classes. Hence it induces a group homomorphism
\[
\overline{\deg}: K_0(\SDS) \longrightarrow \mathbb R_{>0}.
\]

This homomorphism is \emph{surjective}: for each \(K>0\), take any
skeleton of degree \(K\) (for instance, one with a single non-trivial
pair and weights \(\sqrt K\), \(\sqrt K\), which gives \(r=1\)).
Thus the image of \(\overline{\deg}\) is all \(\mathbb R_{>0}\).

The homomorphism \(\overline{\deg}\) is not injective. For instance,
the number of sites \(|Z|\) is multiplicative and invariant under
isomorphism, so it induces a second homomorphism
\(K_0(\SDS)\to\mathbb Q_{>0}\); the skeletons \(\mathcal I\) and a
single orbit with \(r=1\) and \(K=1\) have the same degree but
different numbers of sites. The kernel of \(\overline{\deg}\) is
determined in the companion paper \cite{PaperVIII}: it is the group
of fractions of the multiplicative monoid of symmetric elements of
the group semiring \(\mathbb N_0[\mathbb R]\), and the passage to
\(K_0\) forgets exactly the number of fixed points.
\end{remark}

\begin{corollary}[Reinterpretation of \(K\) and \(C^\ast\)]
\label{cor:K-character}
The degree \(K\) is the \emph{degree character} of the graded
monoidal category \(\SDS\). It provides a surjective homomorphism
\[
\deg: K_0(\SDS) \to \mathbb R_{>0}
\]
that classifies objects into strata \(\SDS_K\). It is not a complete
invariant (Remark~\ref{rem:grothendieck}). The canonical function
\[
C^\ast(K) = \frac{1}{1+\sqrt K}
\]
is then a real function of the degree character, which may be
interpreted probabilistically only in concrete realisations, but in
the abstract theory is a derived invariant.
\end{corollary}

\section{Classification of objects}
\label{sec:classification}

We have defined the objects (weighted dual skeletons), the operations
combining them, and the morphisms relating them. We now answer the
fundamental question: when are two skeletons essentially the same
object? The answer is that the classification up to isomorphism is
given by the invariants identified in the previous sections: the
combinatorial type \((k,f)\), the degree \(K\), and the multiset of
inversion classes of orbit ratios.

\begin{definition}[Isomorphism of skeletons]
\label{def:isomorphism}
Two skeletons \(\mathcal E_1=(Z_1,\sigma_1,\kappa_1)\) and
\(\mathcal E_2=(Z_2,\sigma_2,\kappa_2)\) are \emph{isomorphic} if
there exists a morphism \(\phi:\mathcal E_1\to\mathcal E_2\) that is
bijective. In that case, by Definition~\ref{def:morphism},
\(\phi\) satisfies:
\[
\phi(\sigma_1(z)) = \sigma_2(\phi(z)), \qquad
\kappa_1(z) = \kappa_2(\phi(z)) \quad \forall z\in Z_1,
\]
and its inverse \(\phi^{-1}\) is also a morphism.
\end{definition}

\begin{theorem}[Classification up to isomorphism]
\label{thm:classification}
Two skeletons \(\mathcal E_1\) and \(\mathcal E_2\) are isomorphic if
and only if:
\begin{enumerate}
\item They have the same \emph{combinatorial type}: the number of
  non-trivial pairs \(k\) and the number of fixed points \(f\)
  coincide. Equivalently, \(|Z_1|=|Z_2|\) and the number of orbits of
  each type is the same.
\item They have the same degree \(K\).
\item There exists a bijection between the non-trivial orbits of
  \(\sigma_1\) and those of \(\sigma_2\) such that, for each
  corresponding pair of orbits, the inversion classes of the ratios
  are equal:
  \[
  [r_{O_1}] = [r_{O_2}],
  \qquad [r] = \{r, r^{-1}\}.
  \]
  Equivalently, the multiset of inversion classes
  \(\{[r_1],\dots,[r_k]\}\) is the same in both skeletons.
\end{enumerate}
\end{theorem}

\begin{proof}
We first prove necessity. Suppose
\(\phi:\mathcal E_1\to\mathcal E_2\) is an isomorphism. By
Remark~\ref{rem:consequences}, \(\phi\) preserves degree, so
\(K_1=K_2\). Moreover, \(\phi\) is a bijection commuting with the
involutions, hence it bijects the orbits of \(\sigma_1\) to those of
\(\sigma_2\), preserving the size of each orbit. Thus \(k_1=k_2\) and
\(f_1=f_2\).

For the ratios, pointwise weight preservation gives, for every
\(z\in Z_1\),
\[
r_{1,z} = r_{2,\phi(z)}.
\]
If \(O_1\) is a non-trivial orbit with representative \(z\), and
\(O_2=\phi(O_1)\) is its image, then with the choice of
representatives induced by \(\phi\), one has \(r_{O_1}=r_{O_2}\).
If representatives are changed, the equality may become
\(r_{O_1}=r_{O_2}^{-1}\), but the inversion class \([r]\) is
invariant. Hence the multiset of inversion classes is preserved.

We now prove sufficiency constructively. Suppose \(\mathcal E_1\) and
\(\mathcal E_2\) satisfy the three conditions. We construct an
isomorphism \(\phi\) in several steps:

1.  Use the equality of \(k\) and \(f\) to choose a bijection between
    the non-trivial orbits of \(\sigma_1\) and those of \(\sigma_2\),
    and another bijection between their fixed points.

2.  For each corresponding pair of non-trivial orbits
    \(O_1\subseteq Z_1\) and \(O_2\subseteq Z_2\), choose
    representatives \(z_1\in O_1\) and \(z_2\in O_2\) such that
    \(r_{O_1}=r_{O_2}\). This is possible because the inversion
    classes \([r_{O_1}]\) and \([r_{O_2}]\) are equal: if the class
    equality gives \(r_{O_1}=r_{O_2}^{-1}\), we swap the
    representative in one of the orbits.

3.  Define \(\phi(z_1)=z_2\). For \(\phi\) to commute with the
    involutions, we must set
    \[
    \phi(\sigma_1(z_1)) = \sigma_2(\phi(z_1)) = \sigma_2(z_2).
    \]
    This determines \(\phi\) on the whole orbit \(O_1\).

4.  Check weight equality. By the parametrisation of Section 3,
    \[
    \kappa_1(z_1)=\frac{\sqrt K}{r_{O_1}}
    =\frac{\sqrt K}{r_{O_2}}
    =\kappa_2(z_2),
    \]
    and similarly,
    \[
    \kappa_1(\sigma_1(z_1))=\sqrt K\, r_{O_1}
    =\sqrt K\, r_{O_2}
    =\kappa_2(\sigma_2(z_2)).
    \]
    Hence \(\phi\) preserves weights on the whole orbit.

5.  For fixed points, the equality of \(f\) and \(K\) gives a
    bijection between them, and weight preservation is automatic:
    \(\kappa_1(z_0)=\sqrt K=\kappa_2(\phi(z_0))\), as follows from
    Section 3.

The map \(\phi\) thus defined is a bijection, commutes with the
involutions, and preserves weights. Therefore it is an isomorphism.
\end{proof}

\begin{corollary}[Classification in terms of structural decomposition]
\label{cor:classification-struct}
Up to isomorphism, a weighted dual skeleton is completely determined
by:
\begin{enumerate}
\item The combinatorial type \((k,f)\), where \(k\) is the number of
  non-trivial pairs and \(f\) the number of fixed points.
\item The degree \(K>0\).
\item A multiset of \(k\) inversion classes
  \(\{[r_1],\dots,[r_k]\}\), where each \([r_i]=\{r_i,r_i^{-1}\}\).
\end{enumerate}
This is the most compact and elegant form of the classification.
\end{corollary}

\begin{remark}[Equivariant reformulation]
\label{rem:equivariant}
Let \(C_2=\{1,\tau\}\) act on \(\mathbb R\) by \(\tau t=-t\). For a
skeleton of degree \(K\), put \(\ell(z)=\log r_z\). Since
\(r_{\sigma z}=r_z^{-1}\), the map \(\ell:Z\to\mathbb R\) is
equivariant for the action of \(C_2\) on \(Z\) through \(\sigma\), and
\(\kappa\) is recovered as \(\kappa(z)=\sqrt K\,e^{-\ell(z)}\)
(Proposition~\ref{prop:param}). A map commuting with the involutions
preserves \(\kappa\) if and only if it preserves \(\ell\). Hence
\(\SDS_K\) is isomorphic to the category of finite non-empty
\(C_2\)-sets over \(\mathbb R\), that is, finite \(C_2\)-sets \(X\)
with an equivariant map \(\ell:X\to\mathbb R\), and morphisms
commuting with the maps to \(\mathbb R\). The disjoint union is the
disjoint union of \(C_2\)-sets over \(\mathbb R\), and the Cartesian
product corresponds to \(X\times Y\) with the diagonal action and
\(\ell(x,y)=\ell(x)+\ell(y)\); the symmetric monoidal structure of
Theorem~\ref{thm:monoidal} is thus induced by the commutative group
\((\mathbb R,+)\), whose addition is equivariant.

In this language, Theorem~\ref{thm:classification} is the
decomposition of a finite equivariant set into orbits. There are
three kinds of orbits: a fixed point over \(0\) (a fixed point of
\(\sigma\)), a free orbit over \(\{0\}\) (a non-trivial orbit with
\(r=1\)), and a free orbit over \(\{t,-t\}\) with \(t>0\) (a
non-trivial orbit with \([r]=\{e^t,e^{-t}\}\)). The invariants of the
class are the counts of points over each \(t\) and the number of
fixed points; these are the \emph{marks} of Burnside-ring theory
\cite{Solomon1967,tomDieck1979}, and the semiring of isomorphism
classes, with its arithmetic, is studied from this point of view in
\cite{PaperVIII}. The reformulation also clarifies the categorical
structure: each \(\SDS_K\) has finite coproducts (with the empty
object) and fibre products over a common target, but no terminal
object, since a terminal object would have to be \(\mathbb R\) itself,
which is infinite. The categorical product, made of pairs with the
same value of \(\ell\), is different from the monoidal product of
Theorem~\ref{thm:monoidal}, which adds the values of \(\ell\).
\end{remark}

\begin{example}[Skeletons of degree \(K\) with a single orbit]
\label{ex:one-orbit}
Consider skeletons with \(Z=\{z,\sigma(z)\}\) (a single non-trivial
pair) and degree \(K>0\). By Section 3, the weight is determined by
a single number \(r>0\). Two such skeletons are isomorphic if and
only if they have the same \(K\) and the same class
\([r]=\{r,r^{-1}\}\). Hence the space of isomorphism classes for this
combinatorial type is
\[
\mathbb R_{>0} \times \left(\mathbb R_{>0} \big/ \{r\sim r^{-1}\}\right),
\]
where the quotient is by the involution \(r\mapsto r^{-1}\). This
reflects that swapping the two sites of the orbit inverts \(r\), but
does not change the isomorphism.
\end{example}

\begin{example}[Skeletons of degree \(1\)]
\label{ex:degree-one}
If \(K=1\), the classification simplifies: each non-trivial orbit is
parametrised by a class \([r]=\{r,r^{-1}\}\), and fixed points have
weight \(1\). Hence skeletons of degree \(1\) are classified by the
combinatorial type \((k,f)\) and a multiset of \(k\) inversion
classes. This example shows that degree \(1\) does not imply that the
object is the unit \(\mathcal I\); indeed, there are infinitely many
non-isomorphic objects of degree \(1\) (e.g., varying the number of
orbits or the classes \([r]\)).
\end{example}

\begin{remark}[Connection with probabilistic realisation]
\label{rem:realisation-class}
The classification we have obtained is purely algebraic. When an SDS
is realised probabilistically (for instance, via a biased random walk
with resetting), the degree \(K\) and the orbit ratios \(r_O\) become
physical quantities (such as ruin probability and weights of neutral
distributions). The classification up to isomorphism of the abstract
SDS then translates into a classification of the probabilistic models
realising it. This connection is explored in Section 9, where we
relate the categorical structure to the Fisher--Rao geometry of
Paper~V.
\end{remark}

\section{The response rank}
\label{sec:rank}

In this section we prove that the response rank is the rank of an
explicit matrix built from the amplitudes and the weights
(Theorem~\ref{thm:rango-fuerte}); for a complete SDS this rank is the
number of non-trivial pairs, which is therefore not merely a bound for
the response rank but determines it exactly
(Corollary~\ref{cor:rango-k}). The result rests on a twisted
antisymmetry identity inherited by the response covector from the
duality condition \eqref{eq:skeleton}.

Throughout this section we fix a skeleton \((Z,\sigma,\kappa)\) with
spectral functions \(\{f_\nu\}_{\nu=1}^N\) and amplitudes
\(\{A_\nu\}\), the dual functions being given by
\eqref{eq:amplitudes}; the full-rank condition of
Definition~\ref{def:sds} is not assumed unless stated. We write
\begin{equation}
u_z(\gamma)=\sum_{\nu=1}^N f_\nu(\gamma)\,A_\nu(z),
\qquad
s_z(\gamma)=\sum_{\nu=1}^N f_\nu(\gamma)
            \bigl[A_\nu(z)+B_\nu(z)\bigr]
          =u_z(\gamma)+\kappa(z)\,u_{\sigma z}(\gamma),
\label{eq:us-rank}
\end{equation}
the last equality being immediate from \eqref{eq:amplitudes}. We
assume throughout that \(s_z(\gamma)>0\) for every \(z\) and every
\(\gamma\in(0,1)\), so that the quotient below is well defined; this
holds in the realisations that motivate the theory, where
\(s_z(\gamma)\) is the probability of absorption before the next
reset, starting from \(z\).

\begin{definition}[Coupling functional]
\label{def:coupling}
For a probability vector \(\pi\) on \(Z\) and \(\gamma\in(0,1)\), the
\emph{coupling functional} of \(\mathcal S\) is
\begin{equation}
C(\pi,\gamma)
=\frac{\langle\pi,u(\gamma)\rangle}{\langle\pi,s(\gamma)\rangle},
\qquad
\langle\pi,x\rangle=\sum_{z\in Z}\pi_z x_z .
\label{eq:coupling}
\end{equation}
In the probabilistic realisations of the theory, \(C(\pi,\gamma)\) is
the ruin probability of the process started from the distribution
\(\pi\) with resetting parameter \(\gamma\); see \cite{PaperIII}. Here
it is simply the scalar attached to \(\pi\) by the amplitudes.
\end{definition}

\begin{definition}[Neutral distributions]
\label{def:neutral-rank}
A probability vector \(\pi\) on \(Z\) is \emph{neutral} if
\begin{equation}
\frac{\pi_z}{\pi_{\sigma z}}=r_z
\qquad\text{for every }z\in Z,
\label{eq:neutral-rank}
\end{equation}
with \(r_z\) the orbit ratio of Definition~\ref{def:ratios}. The set
of neutral distributions is denoted \(\Sigma\). Condition
\eqref{eq:neutral-rank} is vacuous at fixed points, where \(r_z=1\);
equivalently, in the form
\(\sqrt{\kappa(z)}\,\pi_z=\sqrt{\kappa(\sigma z)}\,\pi_{\sigma z}\),
it is linear and makes sense on the closed simplex. The structure of
\(\Sigma\) is studied in
Section~\ref{subsec:separatrix-weights}; here we only need that it is
non-empty, which is clear from
Proposition~\ref{prop:param}: assigning an arbitrary positive mass to
each orbit and splitting it in the ratio \(r_O:1\) produces a neutral
distribution.
\end{definition}

The next lemma is what makes the notation consistent: the invariant
\(C^\ast\) of Definition~\ref{def:Cstar}, defined purely algebraically
as a function of the degree, is exactly the value taken by the
coupling functional on neutral distributions --- for every \(\gamma\),
and for every choice of amplitudes.

\begin{lemma}[The invariant is the value of the coupling at neutrality]
\label{lem:Cstar-neutral}
For every neutral \(\pi\in\Sigma\) and every \(\gamma\in(0,1)\),
\begin{equation}
\langle\pi,s(\gamma)\rangle=(1+\sqrt K)\,\langle\pi,u(\gamma)\rangle,
\qquad\text{and hence}\qquad
C(\pi,\gamma)=\frac1{1+\sqrt K}=C^\ast .
\label{eq:Cstar-neutral}
\end{equation}
In particular \(C(\pi,\cdot)\) is constant in \(\gamma\) on \(\Sigma\),
which is the sense in which neutral distributions are insensitive to
the resetting parameter.
\end{lemma}

\begin{proof}
By \eqref{eq:us-rank},
\[
\langle\pi,s\rangle
=\sum_{z}\pi_z u_z+\sum_{z}\pi_z\,\kappa(z)\,u_{\sigma z}.
\]
Reindexing the second sum by \(z\mapsto\sigma z\), which is a
bijection of \(Z\) since \(\sigma\) is an involution, it becomes
\(\sum_z \pi_{\sigma z}\,\kappa(\sigma z)\,u_z\). Now neutrality
\eqref{eq:neutral-rank} gives \(\pi_{\sigma z}=\pi_z/r_z\), while
Proposition~\ref{prop:param} gives \(\kappa(\sigma z)=\sqrt K\,r_z\);
hence
\[
\pi_{\sigma z}\,\kappa(\sigma z)=\sqrt K\,\pi_z
\qquad\text{for every }z\in Z,
\]
including fixed points, where \(r_z=1\) and
\(\kappa(z)=\sqrt K\). Therefore the second sum equals
\(\sqrt K\,\langle\pi,u\rangle\), and
\(\langle\pi,s\rangle=(1+\sqrt K)\langle\pi,u\rangle\). Dividing gives
\(C(\pi,\gamma)=(1+\sqrt K)^{-1}=C^\ast\), independently of
\(\gamma\).
\end{proof}

\begin{remark}
\label{rem:two-Cstars}
Lemma~\ref{lem:Cstar-neutral} is the bridge between the algebraic and
the analytic halves of this paper. In Section 2, \(C^\ast\) was
\emph{defined} as \(1/(1+\sqrt K)\), a function of the degree alone,
with no reference to distributions or amplitudes. The lemma shows
that this algebraic quantity is precisely the value of the coupling
functional at neutrality, so the two readings of \(C^\ast\) agree.
Note that the amplitudes play no role in the statement: the identity
holds for any realisation of the skeleton, which is another instance
of the principle that the invariant belongs to the skeleton and not
to its spectral realisation.
\end{remark}

We may now define the object of study of this section. Fix a neutral
distribution \(\pi^\ast\in\Sigma\) in the interior of the simplex. The
\emph{response covector} at \(\pi^\ast\) is the differential of
\(C(\cdot,\gamma)\) there: for a tangent direction \(x\) with
\(\sum_z x_z=0\),
\begin{equation}
\bigl(d_{\pi^\ast}C(\cdot,\gamma)\bigr)(x)
=\frac{\langle x,u(\gamma)\rangle\,\langle\pi^\ast,s(\gamma)\rangle
       -\langle\pi^\ast,u(\gamma)\rangle\,\langle x,s(\gamma)\rangle}
      {\langle\pi^\ast,s(\gamma)\rangle^{2}}
=\bigl\langle x,\psi(\gamma)\bigr\rangle,
\end{equation}
where, using Lemma~\ref{lem:Cstar-neutral} to replace
\(\langle\pi^\ast,u\rangle/\langle\pi^\ast,s\rangle\) by \(C^\ast\),
\begin{equation}
\psi(\gamma)=\frac{u(\gamma)-C^\ast s(\gamma)}
                  {\langle\pi^\ast,s(\gamma)\rangle}.
\label{eq:psi-rank}
\end{equation}
The \emph{response rank} is
\[
r_{\mathrm{resp}}
=\dim\operatorname{span}\{\psi(\gamma):\gamma\in(0,1)\}.
\]

\begin{remark}[The rank is unambiguous]
\label{rem:rank-welldefined}
Two vectors of \(\mathbb R^{Z}\) differing by a multiple of
\(\mathbf 1=(1,\dots,1)\) define the same linear functional on the
tangent space \(\{x:\sum_z x_z=0\}\), so one might ask whether
\(r_{\mathrm{resp}}\) should be computed in \(\mathbb R^{Z}\) or in
the quotient by \(\mathbf 1\). The two agree. Indeed, by
Lemma~\ref{lem:antisymmetry} below every \(\psi(\gamma)\) lies in the
subspace
\(W=\{v: v_{\sigma z}=-r_zv_z\}\), and \(\mathbf 1\notin W\) since
\(r_z>0\) would force \(1=-r_z\); hence \(W\cap\mathbb R\mathbf 1=0\)
and the quotient map is injective on \(W\). We therefore compute the
rank in \(\mathbb R^{Z}\) without ambiguity.
\end{remark}

\begin{lemma}[Twisted antisymmetry of the response]
\label{lem:antisymmetry}
Define
\begin{equation}
v_z(\gamma)=\sqrt K\,u_z(\gamma)-\kappa(z)\,u_{\sigma z}(\gamma).
\label{eq:v-rank}
\end{equation}
Then
\begin{equation}
\psi(\gamma)=\frac{v(\gamma)}{(1+\sqrt K)\,\langle\pi^\ast,s(\gamma)\rangle},
\label{eq:psiv-rank}
\end{equation}
and the vector \(v\) satisfies, for every \(z\in Z\),
\begin{equation}
v_{\sigma z}=-\,r_z\,v_z .
\label{eq:twisted-rank}
\end{equation}
In particular \(v_z=0\) at every fixed point of \(\sigma\).
\end{lemma}

\begin{proof}
From \eqref{eq:us-rank} and \(C^\ast=(1+\sqrt K)^{-1}\),
\[
u_z-C^\ast s_z
=(1-C^\ast)u_z-C^\ast\kappa(z)u_{\sigma z}
=\frac{\sqrt K\,u_z-\kappa(z)u_{\sigma z}}{1+\sqrt K}
=\frac{v_z}{1+\sqrt K},
\]
which gives \eqref{eq:psiv-rank} after substitution into \eqref{eq:psi-rank}. For
\eqref{eq:twisted-rank}, using \(\kappa(z)\kappa(\sigma z)=K\),
\[
\kappa(z)\,v_{\sigma z}
=\kappa(z)\Bigl[\sqrt K\,u_{\sigma z}-\kappa(\sigma z)u_z\Bigr]
=\sqrt K\,\kappa(z)u_{\sigma z}-K\,u_z
=-\sqrt K\,\bigl[\sqrt K\,u_z-\kappa(z)u_{\sigma z}\bigr]
=-\sqrt K\,v_z,
\]
and since \(\sqrt K/\kappa(z)=r_z\) by
Proposition~\ref{prop:param}, we obtain \eqref{eq:twisted-rank}. At a
fixed point, \(\sigma z=z\) and \(r_z=1\) give \(v_z=-v_z\), hence
\(v_z=0\).
\end{proof}

\begin{remark}
\label{rem:twisted-rank}
Identity \eqref{eq:twisted-rank} says that the response is determined
by a single number in each non-trivial orbit and vanishes identically
at fixed points. The space of vectors satisfying it has dimension
exactly \(k\), recovering the bound \(r_{\mathrm{resp}}\le k\) without
invoking the annihilator of the tangent to \(\Sigma\). Note also that
the minus sign makes the response covector an \emph{odd} object with
respect to the involution, twisted by the weight: this is the same
antisymmetry that in the probabilistic realisation produces the sign
change at \(a/2\).
\end{remark}

Fix a representative \(z_i\) in each of the \(k\) non-trivial
orbits of \(\sigma\), and define the \(N\times k\) matrix
\begin{equation}
M_{\nu i}:=\sqrt K\,A_\nu(z_i)-\kappa(z_i)\,A_\nu(\sigma z_i),
\qquad \nu=1,\dots,N,\quad i=1,\dots,k .
\label{eq:M-matrix}
\end{equation}
Writing \(\mathbf A(z)=(A_1(z),\dots,A_N(z))\in\mathbb R^N\), its
columns are the vectors
\(c_i=\sqrt K\,\mathbf A(z_i)-\kappa(z_i)\,\mathbf A(\sigma z_i)\),
one per orbit.

\begin{theorem}[The response rank is the rank of \(M\)]
\label{thm:rango-fuerte}
Let \((Z,\sigma,\kappa)\) be a weighted dual skeleton with amplitudes
\(\{A_\nu\}_{\nu=1}^N\) and dual functions given by
\eqref{eq:amplitudes}, and assume that the spectral functions
\(f_1,\dots,f_N\) are linearly independent as functions on \((0,1)\).
Then
\begin{equation}
r_{\mathrm{resp}}=\operatorname{rank}M .
\label{eq:rango-fuerte}
\end{equation}
No hypothesis on the rank of the amplitudes is required: the
full-rank condition of Definition~\ref{def:sds} need not hold.
\end{theorem}

\begin{proof}
By Lemma~\ref{lem:antisymmetry}, \(\psi(\gamma)\) is a positive
multiple of \(v(\gamma)\), and \(v\) satisfies
\(v_{\sigma z}=-r_zv_z\) with \(v_z=0\) at every fixed point. Hence
\(v(\gamma)\) is determined by its values at the chosen
representatives, and
\[
r_{\mathrm{resp}}
=\dim\operatorname{span}
 \bigl\{\bigl(v_{z_1}(\gamma),\dots,v_{z_k}(\gamma)\bigr)
        :\gamma\in(0,1)\bigr\}.
\]
From \eqref{eq:us-rank} and \eqref{eq:v-rank},
\[
v_{z_i}(\gamma)=\sum_{\nu=1}^N f_\nu(\gamma)\,M_{\nu i},
\qquad\text{that is}\qquad
\bigl(v_{z_1}(\gamma),\dots,v_{z_k}(\gamma)\bigr)=M^{\mathsf T}f(\gamma),
\]
with \(f(\gamma)=(f_1(\gamma),\dots,f_N(\gamma))\). If the \(f_\nu\)
are linearly independent as functions, the set
\(\{f(\gamma):\gamma\in(0,1)\}\) spans \(\mathbb R^N\): otherwise
there would exist \(c\neq0\) with \(c\cdot f(\gamma)=0\) for every
\(\gamma\), i.e. \(\sum_\nu c_\nu f_\nu\equiv0\). Therefore the span
above is \(\operatorname{im}M^{\mathsf T}\), whose dimension is
\(\operatorname{rank}M\).
\end{proof}

\begin{corollary}[Full rank: the response rank is the number of orbits]
\label{cor:rango-k}
If in addition the \(m\) vectors \(\{\mathbf A(z)\}_{z\in Z}\) are
linearly independent --- that is, for a complete SDS in the sense of
Definition~\ref{def:sds} --- then \(\operatorname{rank}M=k\), and
consequently
\begin{equation}
r_{\mathrm{resp}}=k .
\label{eq:rank-equality}
\end{equation}
\end{corollary}

\begin{proof}
Suppose \(\sum_i\lambda_ic_i=0\). The \(2k\) sites
\(z_1,\sigma z_1,\dots,z_k,\sigma z_k\) are pairwise distinct, so
under the stated hypothesis the vectors \(\mathbf A(z)\) occurring in
that sum are linearly independent. Collecting the coefficient of
\(\mathbf A(z_i)\) gives \(\lambda_i\sqrt K=0\), and since \(K>0\),
\(\lambda_i=0\) for every \(i\). The \(k\) columns are therefore
independent and \(\operatorname{rank}M=k\).
\end{proof}

\begin{remark}[What the general statement adds]
\label{obs:rango-lectura}
Corollary~\ref{cor:rango-k} is Theorem~\ref{thm:rango-fuerte}
followed by one further step: the rank condition on the amplitudes
enters only at the very end, and only to evaluate
\(\operatorname{rank}M\). Three consequences are worth separating.

First, the theorem \emph{identifies} the response rank rather than
computing it in a favourable case. The response rank is the rank of
an explicit \(N\times k\) matrix built from the amplitudes and the
weight; \(k\) is the value it takes when that matrix has full column
rank.

Second, the bound \(r_{\mathrm{resp}}\le k\) holds unconditionally,
since \(M\) has \(k\) columns. It is the equality that requires a
hypothesis, and the hypothesis needed is on \(M\), not on the
amplitudes: independence of the \(\mathbf A(z)\) is sufficient for
\(\operatorname{rank}M=k\), but not necessary.

Third, degeneracy of the amplitudes is not an obstruction to the
theory but a phenomenon within it: amplitudes that fail to be
independent still produce a well-defined response rank, given by the
same formula, and that rank can be strictly smaller than \(k\) --- it
cannot exceed \(N\), for instance. In this sense the result applies
to a class of structures larger than that of complete SDS.
\end{remark}

\begin{remark}[Why the argument survives degeneracy]
\label{obs:antisym-degenerada}
The full-rank hypothesis can be dropped because the twisted
antisymmetry of Lemma~\ref{lem:antisymmetry} is an identity between
the two absorption channels, derived from the duality relation
\eqref{eq:amplitudes} and the constancy of the orbit product
\eqref{eq:skeleton} alone. It involves no independence assumption on
the amplitudes and therefore continues to hold when they degenerate.
What degeneracy affects is only the dimension of the space that the
response covectors actually sweep out, and that dimension is exactly
what \(\operatorname{rank}M\) measures.
\end{remark}

\begin{remark}[Consequences for complete SDS]
\label{rem:rank-consequences}
For a complete SDS, Corollary~\ref{cor:rango-k} has three separate
readings.
\begin{enumerate}
\item \emph{Combinatorial characterisation.} An analytic property
  --- the dimension of the span of response covectors --- coincides
  with an integer counted from the involution. In particular
  \(r_{\mathrm{resp}}\ge2\) if and only if \(k\ge2\): the study of
  high-rank responses is the study of skeletons with many orbits, not
  a condition on the weights.
\item \emph{Categorical invariant.} By
  Theorem~\ref{thm:classification}, \(k\) is invariant under
  isomorphism; hence \(r_{\mathrm{resp}}\) is also invariant. It is
  determined by the isomorphism class, though not sufficient by
  itself to characterise it.
\item \emph{Independence of fibre.} The value of
  \(r_{\mathrm{resp}}\) does not depend on the orbit ratios
  \(r_O\), which are the coordinates of the fibre over \(K\). Moving
  \(r\) translates the separatrix without reorienting it, and the
  rank depends only on the orientation.
\end{enumerate}
Without the full-rank condition, \(r_{\mathrm{resp}}=\operatorname{rank}M\)
depends in general on the amplitudes and on the weights through
\eqref{eq:M-matrix}, and these three readings apply only when
\(\operatorname{rank}M=k\).
\end{remark}

\begin{remark}[On the independence hypothesis]
\label{rem:hypothesis-rank}
The linear independence of the \(\{f_\nu\}\) is necessary and does not
follow from Definition~\ref{def:sds}: if some spectral functions were
dependent, the set \(\{f(\gamma)\}\) would span a proper subspace of
\(\mathbb R^N\) and the rank could drop below \(\operatorname{rank}M\).
In the realisations that motivate the theory --- where, up to a
common nonvanishing factor, \(f_\nu(\gamma)=(1-\gamma)/[1-\lambda_\nu(1-\gamma)]\)
with distinct eigenvalues \(\lambda_\nu\) --- independence is
automatic. It is stated explicitly because it is the only hypothesis
of Theorem~\ref{thm:rango-fuerte}.
\end{remark}

\section{Connection with Fisher--Rao geometry}
\label{sec:geometry}

In Paper~V of this series \cite{PaperV} it was shown that an SDS,
when realised probabilistically, induces a rigid geometry on the
distribution simplex. In particular, the set of neutral distributions
\(\Sigma\) —those for which the ruin probability is independent of
the resetting parameter— is a totally geodesic submanifold for the
Fisher--Rao metric. The aim of this section is to show that this
geometry is not an isolated fact, but a natural manifestation of the
categorical structure we have developed.

\subsection{Review of Paper~V}
\label{subsec:review}

Paper~V establishes the following results for an SDS realised by a
Markovian resetting process:

\begin{enumerate}
\item The set of neutral distributions \(\Sigma\) is given by
  distributions \(\pi\in\Delta_{m-1}\) satisfying, for each
  non-trivial orbit \(O=\{z,\sigma(z)\}\),
  \[
  \sqrt{\kappa(z)}\,\pi_z = \sqrt{\kappa(\sigma(z))}\,\pi_{\sigma(z)}.
  \]
  Equivalently,
  \[
  \frac{\pi_z}{\pi_{\sigma(z)}} = \frac{\sqrt{\kappa(\sigma(z))}}{\sqrt{\kappa(z)}} = r_z.
  \]
  This is exactly the neutrality condition of Definition~\ref{def:neutral-rank}. Fixed points have no constraint.

\item Under the Hellinger embedding
  \[
  \Phi:\Delta_{m-1}^\circ \longrightarrow S^{m-1}_+(2),
  \qquad \Phi(\pi)=2\sqrt{\pi},
  \]
  which is an isometry from the Fisher--Rao metric to the round
  metric of the sphere, the separatrix \(\Sigma\) is identified with
  the intersection of a great subsphere with the positive orthant:
  \[
  \Phi(\Sigma^\circ) = L \cap S^{m-1}_+(2),
  \]
  where \(L\) is the linear subspace defined by
  \[
  \kappa(z)^{1/4} x_z = \kappa(\sigma(z))^{1/4} x_{\sigma(z)}
  \quad\text{for each orbit }O.
  \]
  In particular, \(\Sigma\) is a totally geodesic submanifold.
  Its dimension is \(m-k-1\). Since \(m=f+2k\), where \(f\) is the
  number of fixed points and \(k\) the number of non-trivial orbits,
  we get
  \[
  \dim \Sigma = m-k-1 = f+k-1.
  \]

\item The codimension of \(\Sigma\) in the simplex is \(k\), the number
  of non-trivial pairs of \(\sigma\).
\end{enumerate}

These results are purely geometric and depend only on the weight
structure \((Z,\sigma,\kappa)\). They do not require the choice of
amplitudes \(\{A_\nu\}\). Hence they apply to any weighted dual
skeleton, independently of its spectral realisation.

\subsection{The separatrix as realisation of the weight structure}
\label{subsec:separatrix-weights}

The characterisation of \(\Sigma\) in terms of the orbit ratios \(r_O\)
is the direct connection to our parametrisation of Section 3. Indeed,
the neutrality condition \(\pi_z/\pi_{\sigma(z)}=r_O\) can be
rewritten as:
\[
\pi|_O \propto (r_O, 1),
\]
or equivalently,
\[
\pi|_O \propto (1, r_O^{-1}),
\]
where the proportionality constant \(C_O>0\) may vary from orbit to
orbit. The choice between these forms is a matter of convention; both
are consistent with the definition of \(r_O\) in Section 3.

Thus a neutral distribution is determined by:
\begin{itemize}
\item A positive constant \(C_O\) for each non-trivial orbit \(O\)
  (which determines the total mass of the orbit).
\item A value \(\pi_{z_0}\ge 0\) for each fixed point \(z_0\) (which
  determines its mass).
\item The global normalisation \(\sum_i \pi_i=1\).
\end{itemize}

This is exactly the parameter space of dimension \(f+k-1\).
Moreover, we can prove the coset structure \(\Sigma = \pi^* \oplus H\)
directly and self-contained, using the Aitchison operation
\cite{Aitchison1982,PawlowskyEgozcue2001}.

Let \(\pi^*\in\Sigma^\circ\) be a fixed neutral distribution. For any
\(\pi\in\Sigma^\circ\), define
\[
\rho = \pi \ominus \pi^*
\]
via the Aitchison operation:
\[
\rho_z = \frac{\pi_z/\pi^*_z}{\sum_w \pi_w/\pi^*_w}.
\]
This is the inverse of \(\oplus\) (normalisation of pointwise
product). Then for each non-trivial orbit \(O=\{z,\sigma(z)\}\),
\[
\frac{\rho_z}{\rho_{\sigma(z)}}
=
\frac{\pi_z/\pi^*_z}{\pi_{\sigma(z)}/\pi^*_{\sigma(z)}}
=
\frac{\pi_z/\pi_{\sigma(z)}}{\pi^*_z/\pi^*_{\sigma(z)}}
=
\frac{r_O}{r_O}=1.
\]
Hence \(\rho_z = \rho_{\sigma(z)}\), i.e. \(\rho\in H\), where \(H\) is
the subgroup of \(\sigma\)-symmetric distributions under the Aitchison
operation.

Conversely, if \(\rho\in H\), then for each non-trivial orbit,
\[
\frac{(\pi^*\oplus\rho)_z}{(\pi^*\oplus\rho)_{\sigma(z)}}
=
\frac{\pi^*_z\rho_z}{\pi^*_{\sigma(z)}\rho_{\sigma(z)}}
=
\frac{\pi^*_z}{\pi^*_{\sigma(z)}} \cdot 1
= r_O,
\]
so \(\pi^*\oplus\rho\in\Sigma\). Therefore,
\[
\Sigma = \pi^* \oplus H.
\]
This coset structure is a universal property of the separatrix,
independent of the probabilistic realisation, and follows directly
from the neutrality condition.

\subsection{Surjective morphisms induce maps between separatrices}
\label{subsec:morphisms-separatrices}

Let \(\phi:\mathcal E_1\to\mathcal E_2\) be a morphism of skeletons.
By Remark~\ref{rem:consequences}, \(\phi\) preserves the orbit ratios
in the following sense: for every \(z\in Z_1\),
\[
r_{1,z} = r_{2,\phi(z)}.
\]
Recall from Section 3 that \(r_z=1\) for fixed points. Moreover, by
commutation with the involutions, a fixed point of \(\sigma_1\) is
necessarily sent to a fixed point of \(\sigma_2\):
\[
\sigma_2(\phi(z)) = \phi(\sigma_1(z)) = \phi(z).
\]
Thus the equality \(r_{1,z}=r_{2,\phi(z)}\) holds for every
\(z\), including fixed points.

Given a neutral distribution \(\pi\in\Sigma_1\) in the simplex of
\(\mathcal E_1\), define its \emph{pushforward}
\(\phi_\#\pi\) as the distribution on \(Z_2\) given by
\[
(\phi_\#\pi)_{z'} = \sum_{\phi(z)=z'} \pi_z.
\]

\begin{proposition}[Preservation of neutrality for surjective morphisms]
\label{prop:pushforward}
If \(\phi:\mathcal E_1\to\mathcal E_2\) is a \emph{surjective}
morphism, then
\[
\phi_\#(\Sigma_1) \subseteq \Sigma_2.
\]
That is, the image of a neutral distribution is neutral.
\end{proposition}

\begin{proof}
Let \(\pi\in\Sigma_1\). For a site \(z'\in Z_2\), its mass in
\(\phi_\#\pi\) is the sum of the masses of its preimages. We must
verify the neutrality condition for each non-trivial orbit
\(O_2\) of \(\sigma_2\).

Take a representative \(z'_0\in O_2\). Then
\[
(\phi_\#\pi)_{z'_0} = \sum_{\phi(z)=z'_0} \pi_z,
\qquad
(\phi_\#\pi)_{\sigma_2(z'_0)} = \sum_{\phi(z)=\sigma_2(z'_0)} \pi_z.
\]
Since \(\phi\) commutes with the involutions, if \(\phi(z)=z'_0\)
then \(\phi(\sigma_1(z))=\sigma_2(z'_0)\). Hence the sum over the
preimages of \(\sigma_2(z'_0)\) is the sum of the weights of the
images under \(\sigma_1\) of the preimages of \(z'_0\).
Using neutrality of \(\pi\) in each orbit of \(\sigma_1\),
\[
\pi_{\sigma_1(z)} = \frac{\pi_z}{r_{1,z}}.
\]
But \(r_{1,z}=r_{2,z'_0}\) by preservation of ratios. Summing over
all preimages:
\[
(\phi_\#\pi)_{\sigma_2(z'_0)}
= \sum_{\phi(z)=z'_0} \pi_{\sigma_1(z)}
= \sum_{\phi(z)=z'_0} \frac{\pi_z}{r_{2,z'_0}}
= \frac{1}{r_{2,z'_0}} (\phi_\#\pi)_{z'_0}.
\]
Thus
\[
\frac{(\phi_\#\pi)_{z'_0}}{(\phi_\#\pi)_{\sigma_2(z'_0)}} = r_{2,z'_0},
\]
which is the neutrality condition in \(\mathcal E_2\). Hence
\(\phi_\#\pi\in\Sigma_2\).
\end{proof}

\begin{remark}
Surjectivity is not necessary for preserving neutrality as such; it is
only needed to ensure that the pushforward of an interior
distribution remains interior (i.e., has no zero components).
Therefore, within the open simplex we restrict the proposition to
surjective morphisms.

If one works with the closed separatrix, allowing zero components,
the preservation of neutrality extends to arbitrary morphisms. In that
case, the neutrality condition is formulated in its linear form:
\[
\sqrt{\kappa(z)}\,\pi_z = \sqrt{\kappa(\sigma(z))}\,\pi_{\sigma(z)},
\]
and the proof is analogous, without dividing by \(\pi_{\sigma(z)}\).
This avoids any problems with quotients when both masses are zero.
\end{remark}

\begin{remark}
The pushforward \(\phi_\#\) is linear as a map between spaces of
probability measures, but it is not a homomorphism for the Aitchison
operation \(\oplus\). Therefore we do not claim that \(\phi_\#\)
respects the coset structure \(\Sigma = \pi^*\oplus H\) in the group
sense. Compatibility of the pushforward with the coset structure, if
desired, would require a specific proof not addressed in this section.
\end{remark}

\subsection{The Cartesian product and tensor product of distributions}
\label{subsec:product-tensor}

Given two skeletons \(\mathcal E_1=(Z_1,\sigma_1,\kappa_1)\) and
\(\mathcal E_2=(Z_2,\sigma_2,\kappa_2)\), their Cartesian product
\(\mathcal E_1\times\mathcal E_2\) has site set \(Z_1\times Z_2\),
involution \(\sigma_1\times\sigma_2\), and orbit ratios (by
Proposition~\ref{prop:product})
\[
r_{(z,w)} = r_{1,z} \cdot r_{2,w}.
\]

Define the \emph{tensor product} of distributions
\(\pi_1\in\Delta(Z_1)\) and \(\pi_2\in\Delta(Z_2)\) as the
distribution on \(Z_1\times Z_2\) given by
\[
(\pi_1 \otimes \pi_2)_{(z,w)} = \pi_{1,z} \cdot \pi_{2,w}.
\]
This is the product distribution (in the sense of probability
measure) on the Cartesian product of the site sets.

\begin{proposition}[Tensor product preserves neutrality]
\label{prop:tensor-neutrality}
If \(\pi_1\in\Sigma_1\) and \(\pi_2\in\Sigma_2\), then
\[
\pi_1 \otimes \pi_2 \in \Sigma_{12},
\]
where \(\Sigma_{12}\) is the separatrix of the Cartesian product
\(\mathcal E_1\times\mathcal E_2\).
\end{proposition}

\begin{proof}
For every \((z,w)\in Z_1\times Z_2\),
\[
\frac{(\pi_1\otimes\pi_2)_{(z,w)}}
{(\pi_1\otimes\pi_2)_{(\sigma_1(z),\sigma_2(w))}}
=
\frac{\pi_{1,z}}{\pi_{1,\sigma_1(z)}}
\cdot
\frac{\pi_{2,w}}{\pi_{2,\sigma_2(w)}}
=
r_{1,z} \cdot r_{2,w}
=
r_{(z,w)}.
\]
Hence \(\pi_1\otimes\pi_2\) satisfies the neutrality condition for
\(\mathcal E_1\times\mathcal E_2\).
\end{proof}

\begin{remark}
The map
\[
\Sigma_1 \times \Sigma_2 \longrightarrow \Sigma_{12},
\qquad
(\pi_1,\pi_2) \longmapsto \pi_1\otimes\pi_2
\]
is a natural map induced by the Cartesian product of skeletons.
However, in general it is not surjective: a neutral distribution on
\(Z_1\times Z_2\) may have correlations between the coordinates
\(z\) and \(w\), and may not factorise as a tensor product. The image
of this map is the subset of neutral distributions that are
\emph{independent} in the sense that the coordinates \(z\) and \(w\)
are independent under the probability measure.

This observation is analogous to the structure of orbits of the
Cartesian product seen in Remark~\ref{rem:param-interp}: the product
of two non-trivial orbits splits into two orbits, and neutral
distributions can assign independent masses to each of them, giving
rise to correlations.
\end{remark}

\begin{remark}[Aitchison and tensor product]
\label{rem:aitchison-tensor}
The Aitchison operation \(\oplus\) \cite{Aitchison1982,PawlowskyEgozcue2001}, defined by
\[
(\pi\oplus\rho)_z = \frac{\pi_z \rho_z}{\sum_w \pi_w \rho_w},
\]
is an operation requiring \(\pi\) and \(\rho\) to be defined on the
same set of coordinates. The tensor product \(\otimes\) is a
different operation, combining distributions defined on different
coordinate sets. Both operations are important, but should not be
confused. The first gives a group structure to the open simplex; the
second is the natural operation associated with the Cartesian product
of skeletons.
\end{remark}

\subsection{The response rank as a categorical invariant}
\label{subsec:rank-invariant}

In Section 8 we proved that the response rank \(r_{\mathrm{resp}}\)
is the rank of the matrix \eqref{eq:M-matrix}, and that for a complete
SDS it is exactly equal to the number of non-trivial pairs \(k\)
(Corollary~\ref{cor:rango-k}):
\[
r_{\mathrm{resp}} = k.
\]
By the classification theorem (Section 7), \(k\) is an invariant of
the isomorphism class of a skeleton (together with \(f\), \(K\), and
the multiset of classes \([r]\)). Hence \(r_{\mathrm{resp}}\) is
determined by the isomorphism class: if two skeletons are isomorphic,
they have the same \(r_{\mathrm{resp}}\). Conversely, equality of
\(r_{\mathrm{resp}}\) is not sufficient to characterise the
isomorphism; all the invariants of the classification are needed.

This equality reinforces the idea that the Fisher--Rao geometry and
the response of the system are manifestations of the underlying
algebraic structure. For a complete SDS the rank is not a number
depending on the continuous parameters \(r_O\) or on the spectral
realisation; it is
determined exclusively by the combinatorics of the involution
\(\sigma\), which is one of the categorical invariants of the SDS.

\subsection{Conclusion: geometry as a natural realisation of the algebraic structure}
\label{subsec:conclusion-geometry}

We have seen that:
\begin{itemize}
\item The separatrix \(\Sigma\) is the set of distributions satisfying
  the orbit ratios \(r_O\), which are exactly the weight parameters
  of a skeleton.
\item Surjective morphisms induce maps between separatrices that
  preserve neutrality.
\item The Cartesian product of SDS induces a natural map
  \(\Sigma_1\times\Sigma_2\to\Sigma_{12}\) via the tensor product of
  distributions, which preserves neutrality.
\item For a complete SDS, the response rank \(r_{\mathrm{resp}}\) is
  a categorical invariant, determined by the isomorphism class.
\end{itemize}

Thus the Fisher--Rao geometry of Paper~V provides a \emph{natural
geometric realisation} of part of the underlying algebraic structure.
We do not claim that the geometry is a faithful representation of the
entire category \(\SDS\), but rather that certain structural aspects
—the orbit ratios, the separatrix, neutrality, and the induced
operations— are coherently reflected in the distribution simplex.
This connection suggests that the geometric rigidity observed in
probabilistic resetting models is not merely an accidental feature of
those models, but reflects structural relations already present at
the abstract level of the SDS.

\section{Future work and open questions}
\label{sec:future}

The theory developed in this paper establishes the foundations of an
algebraic and categorical structure for SDS. However, as is often the
case with foundational works, the answers obtained raise as many
questions as they close. In this section we point out the directions
we consider most promising and that lie outside the scope of this
paper.

\subsection{The Grothendieck group of the monoidal category}
\label{subsec:grothendieck-future}

In Section 6 we proved that the degree \(\deg\) induces a surjective
homomorphism from the Grothendieck group \(K_0(\SDS)\) onto
\(\mathbb R_{>0}\), and observed that it is not injective. The
complete answer is given in the companion paper \cite{PaperVIII}:
the class of a skeleton is determined by its degree, the symmetric
element \(\sum_z x^{\log r_z}\) of the group semiring
\(\mathbb N_0[\mathbb R]\) and its number of fixed points; the
multiplicative monoid of classes is not cancellative, and its
arithmetic (atoms, non-unique factorization, elasticity) can be
described in detail. In particular, the invariant
\(N(\mathcal E)=\log|Z|\) suggested above descends to the Grothendieck
group, since \(|Z|\) is multiplicative.

\subsection{Projective morphisms and fibred categories}
\label{subsec:projective}

In this paper we have worked with morphisms that exactly preserve
weights: \(\kappa_1(z)=\kappa_2(\phi(z))\). A natural generalisation
is to allow a scale factor \(\lambda>0\):
\[
\kappa_1(z) = \lambda\, \kappa_2(\phi(z)).
\]
In that case, the degree transforms as
\[
K_1 = \lambda^2 K_2,
\]
while the orbit ratios are preserved:
\[
r_{1,z}=r_{2,\phi(z)}.
\]
These \emph{projective morphisms} would connect objects of different
strata \(\SDS_K\) and give rise to a fibred category over
\(\mathbb R_{>0}\). The fibre over each \(K\) would be the
subcategory \(\SDS_K\) we have studied, and projective morphisms
would provide changes of fibre. This structure might have a natural
geometric interpretation in terms of the Fisher--Rao metric, where the
scale factor \(\lambda\) would correspond to a homothety in the
Hellinger embedding. The study of this fibred category and its
functorial properties is a natural continuation.

\subsection{Closure of the realisable class and realisation theory}
\label{subsec:realisable}

One of the original goals of this line of research was to understand
the SDS arising from Markovian resetting processes, in particular the
biased random walk with geometric resetting. Paper~V showed that
these SDS induce a rigid geometry on the simplex, but did not
characterise which SDS are realisable by some process.

Now that we have defined the operations of Cartesian product and
disjoint union, we can ask: is the subclass of realisable SDS closed
under these operations? In other words, if \(\mathcal S_1\) and
\(\mathcal S_2\) are realisable by resetting processes, is
\(\mathcal S_1\times\mathcal S_2\) (or \(\mathcal S_1\sqcup\mathcal S_2\))
also realisable? The answer would require a realisation theory for
SDS, which is an open problem. A first step would be to study the
biased random walk on products of intervals, where the Cartesian
product of SDS corresponds to a random walk on the product graph with
geometric resetting. An affirmative answer in this case would provide
a powerful family of examples.

\subsection{Representation theory and degree character}
\label{subsec:representations}

The graded monoidal category \(\SDS\) suggests a theory of
representations in which the degree \(K\) acts as a character.
Specifically, a monoidal functor
\[
\rho:\SDS \longrightarrow \mathbf{Vect}
\]
to the category of graded vector spaces could encode the invariant
\(K\) as the degree of the representation space. In this context,
\(C^\ast(K)=1/(1+\sqrt K)\) would be a scalar function of the
character, analogous to the dimension of a representation in group
theory. This line of work connects directly with the spectral theory
of Papers~I--IV, where the amplitudes \(\{A_\nu\}\) provide a linear
representation of the skeleton in the space of modes. A systematic
development of this connection promises to unify the two sides of the
theory: the algebraic (monoidal category) and the analytic (spectrum
of operators). A first instance is the linearisation
\(Z\mapsto\mathbb R^Z\), which is monoidal for the Cartesian product;
its natural invariants are the marks of
Remark~\ref{rem:equivariant} \cite{PaperVIII}.

\subsection{Limits and colimits}
\label{subsec:limits}

By Remark~\ref{rem:equivariant}, each stratum \(\SDS_K\) has finite
coproducts and fibre products over a common target, but no terminal
object, so it is not a topos. A natural question is which further
limits and colimits exist once projective morphisms
(Section~\ref{subsec:projective}) are allowed, and whether the
resulting category admits a useful universal description.

\subsection{Connections with information theory and physics}
\label{subsec:information}

Finally, we point out a more speculative but promising direction. The
Fisher--Rao geometry of Paper~V and the Hellinger embedding that
realises it are the same objects that appear in quantum information
theory, where Uhlmann fidelity and the Bures metric provide an
analogous structure. In that context, the invariant \(C^\ast\) could
be interpreted as a quantum fidelity between states, and the
separatrix \(\Sigma\) as a submanifold of states indistinguishable
under certain operations. This connection, though still speculative,
suggests that the SDS might have a realisation within quantum
information theory, where the degree \(K\) would be a decoherence or
mixing parameter. Exploring this possibility is one of the most
ambitious directions opened by this work.

\subsection{Conclusion of the section}

Paper~VI has established a solid foundation for the algebraic and
categorical study of SDS. The questions we have enumerated —the
Grothendieck group (settled in \cite{PaperVIII}), projective
morphisms, realizability, representation theory, categorical limits,
and connections with physics— delineate a broad and coherent research programme. Each of
them can be addressed separately, and all rest on the architecture we
have built. We hope that this work serves as a starting point for
future explorations of this mathematical structure, which emerged
from a concrete probabilistic problem but, as we have seen, has a
much broader life of its own.

\section*{Acknowledgements}

The author thanks Tony Newton for stimulating discussions, insightful
comments, and valuable contributions throughout the development of
this research program.


\end{document}